\documentclass[preprint,12pt,authoryear]{elsarticle}
\usepackage{graphicx,amsmath,amsthm,amsfonts}
\newtheorem{definition}{Definition}
\newtheorem{theorem}{Theorem}
\newtheorem{proposition}{Proposition} 
\newtheorem{lemma}{Lemma}
\newtheorem{example}{Example}
\newtheorem{corollary}{Corollary}
\newtheorem{remark}{Remark}
\newcommand{\D}[0]{\mathrm{d}}

\begin{document}
\begin{frontmatter}



\title{A new summability method for divergent series}


\author{Ibrahim M. Alabdulmohsin}

\address{King Abdullah University of Science and Technology (KAUST),\\
Computer, Electrical and Mathematical Sciences \& Engineering Division\\ 
Thuwal, 23955-6900, Saudi Arabia}

\begin{abstract}
The theory of summability of divergent series is a major branch of mathematical analysis that has found important applications in engineering and science. It addresses methods of assigning natural values to divergent sums, whose prototypical examples include the Abel summation method, the Ces\'aro means, and Borel summability method. In this paper, we introduce a new summability method for divergent series, and derive an asymptotic expression to its error term. We show that it is both regular and linear, and that it arises quite naturally in the study of local polynomial approximations of analytic functions. Because the proposed summability method is conceptually simple and can be implemented in a few lines of code, it can be quite useful in practice for  \emph{computing} the values of {divergent} sums. 

\end{abstract}

\begin{keyword}
Summability theory \sep Divergent series \sep Abel summation  \sep Ces\'aro summability

\PACS 40C05 \sep 65B10 \sep 65B15
\end{keyword}

\end{frontmatter}

\section{Introduction}\label{intro}
At the turn of the 18\textsuperscript{th} century, a mathematician named Guido Grandi (1671-1742) proposed a mathematical argument for creationism. He looked into the divergent series $\sum_{k=0}^\infty (-1)^k$ and argued that its value was both equal to zero and to non-zero \emph{simultaneously} because both statements could be equally justified. On one hand, grouping terms into $(1-1)+(1-1)+\cdots$ suggested a value of zero for the infinite sum. On the other hand, the series itself arose in a well-known Maclaurin expansion of the function $f(x)=(1+x)^{-1}$, which suggested a value of $1/2$ for the infinite sum. Since the same mathematical object could be equally assigned a zero value (nothing) and a non-zero value (something), Grandi argued that creation out nothing was mathematically justifiable \citep{kline1983euler,kowalenko2011euler}. Nearly a century later, Niels Abel would describe divergent series as the \lq\lq work of the devil'' and declare it \lq\lq shameful'' for any mathematician to base any argument on them \citep{HardyDiverg,kowalenko2011euler,tucciarone1973development,varadarajan2007euler}. 

Divergent series are an inconvenient truth. They arise quite frequently in many branches of mathematics and sciences, such as asymptotic analysis, analytic number theory, Fourier analysis,  quantum field theory, and dynamical systems \citep{caliceti2007useful,kowalenko2011euler,rubel1989editor,tucciarone1973development,varadarajan2007euler}, but their divergence makes them difficult to interpret and analyze. As a result, they stirred an intense debate in the mathematical community for many centuries. Throughout the 17\textsuperscript{th} and 18\textsuperscript{th} centuries, for instance, mathematicians used divergent series regularly, albeit cautiously. They advocated a \emph{formal} interpretation, where a series  generally assumed the value of the algebraic expression from which it was derived \citep{ferraro1999first,varadarajan2007euler}. Euler, in particular, believed that every series had a unique value, and that divergence was a mere \emph{artificial} limitation \citep{ferraro1999first,HardyDiverg,kline1983euler,varadarajan2007euler}. In the 19\textsuperscript{th} century, however, the advent of mathematical rigor had led the most prominent mathematicians of the time, such as Cauchy and Weierstrass, to forbid the use of divergent series entirely. Consequently, little (if any) work was published on divergent series from 1830 to 1880 \citep{ferraro1999first,HardyDiverg,kline1983euler,kowalenko2011euler,tucciarone1973development,varadarajan2007euler}.

Around the turn of the 20\textsuperscript{th} century, nevertheless, a Hegelian synthesis between the two opposing views came to light. It was initiated by Ernesto Ces\'aro (1859-1906), who provided the first modern definition of divergent series that placed their study on a rigorous footing \citep{ferraro1999first,HardyDiverg,tucciarone1973development}. Ces\'aro's definition was an averaging method, which was later generalized independently by N\"orlund and Voronoi \citep{HardyDiverg}. Ramanujan would share a similar view later by describing the value of an infinite sum as a \lq\lq center of gravity'' \citep{Berndt6}. The study of divergent series included contributions from prominent mathematicians such as Frobenius, Borel, Hardy, Ramanujan, and Littlewood, and the name \emph{summability theory} was coined to denote this emerging branch of mathematical analysis.

Perhaps, the most fundamental question in summability theory is how to \emph{interpret} divergent series, such as the Grandi series mentioned earlier. One natural method, first postulated in a limited form by Hutton and later expanded by H\"older, Ces\'aro, N\"orlund, and Voronoi among others, is the \emph{averaging} interpretation. For instance, in one of the earliest and simplest definitions, the \emph{limit} of a sequence $(s_k)_{k=1,2,\ldots}$ is \emph{defined} by the limit of the sequence of averages $(s_k+s_{k+1})/2$. This definition, often referred to as Hutton's summability method, is a reasonable definition of divergent sums because it possesses many useful properties. First, it is \emph{regular}, i.e. if the original sequence $(s_k)_{k=1,2,\ldots}$  had a limit, then the new definition also agrees with that limit. It is also \emph{linear} and \emph{stable} (see \citep{HardyDiverg} for a definition of those terms). Importantly, it has a non-trivial \emph{power} because it can assign values to some series that do not converge in the classical sense of the word. For example, applying this summability method to the Grandi series $1-1+1-1+\cdots$ yields a value of $\frac{1}{2}$, which is the same value assigned to this series using countless other arguments. 

Besides the \emph{averaging} interpretation, some summability methods were proposed that appeal to \emph{continuity} as an argument for defining divergent series. The most prototypical example in this regard is the Abel summation method, which was named after Abel's continuity theorem \citep{HardyDiverg,tucciarone1973development}. It assigns to a series $\sum_{k=0}^\infty a_k$ the value $\lim_{x\to 1^-}\,\sum_{k=0}^\infty a_kx^k$ if the limit exists. Euler used this approach quite extensively and called it \lq\lq the generating function method'' \citep{kowalenko2011euler,varadarajan2007euler}. Poisson also used it frequently in summing Fourier series \citep{HardyDiverg,tucciarone1973development}. For the Grandi series, this approach also yields a value of $\frac{1}{2}$. This agreement between the Abel summation method and averaging methods is not a coincidence. In fact, it can be shown that the Abel summation method is itself an averaging method, and that it is always consistent with, but is more powerful than, all N\"orlund means, including Ces\'aro and Hutton's summability methods \citep{HardyDiverg,tauberianTheory}.

Other major approaches for defining divergent series have been proposed and are used extensively in the literature as well. These include definitions that were originally conceived of as methods for accelerating series convergence, such as the Euler summation method \citep{caliceti2007useful,Cohen2000,HardyDiverg,lagarias2013euler}.  They also include methods that are based on analytic continuation, such as the zeta function regularization method. Moreover, they include the Mittag-Leffler summability method as well as the Lindel\"of summability method, both summing power series in the Mittag-Leffler star of the corresponding function \citep{HardyDiverg}. Summability methods such as the latter ones, which assign to a power series the value of the generating function, are called \emph{analytic} \citep{moroz1990novel}.

In this paper, a new summability method for divergent series is introduced. We will show how it arises quite naturally in the study of local polynomial approximations, and derive an asymptotic expression to its error term. In particular, we will prove that if it converges, then it converges algebraically (logarithmically). When applied to the geometric series $\sum_{k=0}^\infty x^k$ with $x\in\mathbb{R}$, we will prove that the new summability method assigns to this series the value $(1-x)^{-1}$ as long as $-\kappa<x<1$, where $\kappa\approx 3.5911$ is the solution to $\kappa\log\kappa -\kappa = 1$. For other values of $x$ that lie outside this domain, it does not assign any value to the geometric series. Because $\kappa>1$, the proposed summability method is strictly stronger than the Abel summation method. Also, we will show that it is a linear regular matrix summability method that can also be interpreted as an averaging method. It is conceptually quite simple and can be implemented in a few lines of code. Hence, it can be quite useful in practice for \emph{computing} the values of {divergent} sums. Finally, we will demonstrate how the proposed summability method works quite well, even for asymptotic series that diverge everywhere such as the Euler-Maclaurin summation formula. 

\section{The Summability Method} 
We will begin by introducing the new summability method, and show how it arises quite naturally in the study of local polynomial approximations. We analyze the summability method afterward.

\subsection{Preliminaries}
\begin{definition}[The $\chi$-Sum]\label{chi_sum_def}
Let $(a_k)_{k=0,1,2,\ldots}\in\mathbb{C}^\infty$ be an infinite sequence of complex numbers. Then, the sequence $(a_k)_{k=0,1,2,\ldots}$ will be called $\chi$-summable if the following limit exists:
\begin{equation}\label{chi_sum_def}
V=\lim_{n\to\infty}\;\Big\{\sum_{k=0}^n \chi_n(k)\,a_k\Big\},
\end{equation}
where $\chi_n(k) = \prod_{j=1}^k \big(1-\frac{j-1}{n}\big)$ and $\chi_n(0)=1$. In addition, the limit $V$, if it exists, will be called the $\chi$-sum of the infinite sequence $(a_k)_{k=0,1,2,\ldots}$.
\end{definition}

\begin{remark}
With some abuse of terminology, we will sometimes say that the series $\sum_{k=0}^\infty a_k$ is $\chi$-summable when we actually mean that the sequence $(a_k)_{k=1,2,\ldots}$ is $\chi$-summable according to Definition \ref{chi_sum_def}.
\end{remark}

\begin{definition}[The $\chi$-Limit]\label{chi_seq_limit_def}
Let $(s_k)_{k=0,1,2,\ldots}\in\mathbb{C}^\infty$ be an infinite sequence of complex numbers. Then, the $\chi$-limit of the sequence $(s_k)_{k=1,2,\ldots}$ is defined by the following limit if it exists:
\begin{equation}
\lim_{n\to\infty} \Big\{\frac{\sum_{k=0}^n\,p_n(k)\,s_k}{\sum_{k=0}^n p_n(k)}\Big\}, 
\end{equation} 
where $p_n(k)=k\,\chi_n(k)$.
\end{definition}

Later, we will show that both Definition \ref{chi_sum_def} and Definition \ref{chi_seq_limit_def} are equivalent to each other. That is, the $\chi$-limit of the sequence of partial sums:
\begin{equation*}
\big(\sum_{k=0}^0\,a_k, \,\sum_{k=0}^1\,a_k,\,\sum_{k=0}^2\,a_k,\ldots \big)
\end{equation*}
is equal to the $\chi$-sum of the infinite sequence $(a_k)_{k=1,2,\ldots}$. This shows that the summability method in Definition \ref{chi_sum_def}, henceforth referred to as $\chi$, is indeed an averaging method, and allows us to prove additional useful properties. Before we do that, we first show how the summability method $\chi$ arises quite naturally in local polynomial approximations. 

\subsection{Summability and the Taylor Series Approximation}\label{sum_taylor_derivation}
As a starting point, suppose we have a function $f(x)$ that is \emph{n}-times differentiable at a particular point $x_0$ and let $h\ll 1$ be a small chosen small step size such that we wish to approximate the value of the function $f(x_0 + nh)$ for an arbitrary value of $n$ using \emph{solely} the local information about the behavior of $f$ at $x_0$. Using the notation $f_j^{(k)} \doteq f^{(k)}(x_0+jh)$ to denote the $k$-th derivative of $f$ at the point $x_0+jh$, we know that the following approximation holds when $h\ll 1$:
\begin{equation}\label{Claim411_1} 
f_j^{(k)} \approx f_{j-1}^{(k)} + f_{j-1}^{(k+1)}h
\end{equation}
In particular, we obtain the following approximation that can be held with an arbitrary accuracy for a sufficiently small step size $h$: 
\begin{equation}\label{Claim411_2} 
f_n \approx f_{n-1} + f_{n-1}^{(1)}h
\end{equation} 
However, the approximation in (\ref{Claim411_1}) can now be applied to the right-hand side of (\ref{Claim411_2}). In general, we can show by induction that a repeated application of (\ref{Claim411_2}) yields the following general formula:
\begin{equation}\label{Claim411_4}
f_n \approx \sum_{k=0}^n \binom{n}{k} f_0^{(k)} h^k
\end{equation}

To prove that (\ref{Claim411_4}) holds, we first note that a base case is established for $n=1$ in  (\ref{Claim411_2}). Suppose that it holds for $n < m$, we will show that this inductive hypothesis implies that (\ref{Claim411_4}) also holds for $n = m$. First, we note that if (\ref{Claim411_4}) holds for $n<m$, then we have:
\begin{equation}\label{Claim411_5}
f_m \approx f_{m-1} + f_{m-1}^{(1)} h \approx \sum_{k=0}^{m-1} \binom{m-1}{k} f_0^{(k)} h^k + \sum_{k=0}^{m-1} \binom{m-1}{k} f_0^{(k+1)} h^{k+1}
\end{equation}
In (\ref{Claim411_5}), the second substitution for $f_{m-1}^{(1)}$ follows from the same inductive hypothesis because $f_{m-1}^{(1)}$ is simply another function that can be approximated using the same inductive hypothesis. Equation (\ref{Claim411_5}) can, in turn, be rewritten as:
\begin{equation}\label{Claim411_6}
f_m \approx \sum_{k=0}^{m-1} \Big[\binom{m-1}{k} + \binom{m-1}{k-1}\Big] f_0^{(k)} h^k +f_0^{(m)} h^m
\end{equation}
Upon using the well-known recurrence relation for binomial coefficients, i.e. Pascal's rule, we obtain (\ref{Claim411_4}), which is precisely what is needed in order to complete the proof by induction of that approximation. 

In addition, (\ref{Claim411_4}) can be rewritten as given in (\ref{Claim411_7}) below using the substitution $x=x_0 + nh$.
\begin{equation}\label{Claim411_7}
f(x) \approx \sum_{k=0}^n \binom{n}{k} f^{(k)}(x_0) \frac{(x-x_0)^k}{n^k}
\end{equation}
Next, the binomial coefficient can be expanded, which yields:
\begin{equation}\label{claim411_final_eq}
f(x) \approx\sum_{k=0}^n \chi_n(k)\,\frac{f^{(k)}(x_0)}{k!} (x-x_0)^k
\end{equation}

Since $x=x_0 + nh$, it follows that increasing $n$ while holding $x$ fixed is equivalent to choosing a smaller step size $h$. Because the entire proof is based solely on the linear approximation recurrence given in (\ref{Claim411_1}), which is an asymptotic relation as $h\to 0$, the approximation error in (\ref{claim411_final_eq}) will typically vanish as $n\to\infty$. Intuitively, this holds because the summability method $\chi$ uses a sequence of first-order approximations for the function $f(x)$ in the domain $[x_0,\, x]$, which is similar to the Euler method. However, contrasting the expression in (\ref{claim411_final_eq}) with the classical Taylor series expansion for $f(x)$ gives rise to the $\chi$ summability method in Definition \ref{chi_sum_def}. 

Interestingly, the summability method $\chi$ can be stated succinctly using the language of \emph{symbolic methods}. Here, if we let $D$ be the \emph{differential operator}, then Definition \ref{chi_sum_def} states that:
\begin{equation}\label{summMethodCFD}
f(x)=\lim_{n\to\infty} \Big(1+\frac{xD}{n}\Big)^n,
\end{equation}
where the expression is interpreted by expanding the right-hand side into a power series of $x$ and $D^k$ is interpreted as $f^{(k)}(0)$. On the other hand, a Taylor series expansion states that $f(x)=e^{xD}$. Of course, both expressions are equivalent if $D$ were an ordinary number, but they differ when $D$ is a functional operator, as illustrated above.

\subsection{Properties of the Summability Method} 
Next, we prove that both Definition \ref{chi_sum_def} and Definition \ref{chi_seq_limit_def} are equivalent to each other. In particular, the summability method  $\chi$ can be interpreted as an averaging method on the sequence of partial sums. In addition, we prove that it is both regular and linear.

\begin{proposition}\label{XiLimitSeq}
Let $(s_k)_{k=0,1,\ldots}\in\mathbb{C}^\infty$ be a sequence of partial sums $s_k = \sum_{j=0}^k\,a_j$ for some $(a_k)_{k=0,1,2,\ldots}\in\mathbb{C}^\infty$. Then the $\chi$-limit of $(s_k)_{k=1,2,\ldots}$ given by Definition \ref{chi_seq_limit_def} is equal to the $\chi$-sum of $(a_k)_{k=1,2,\ldots}$ given by Definition \ref{chi_sum_def}. 
\end{proposition}
\begin{proof} 
Before we prove the statement of the proposition, we first prove the following useful fact: 
\begin{equation}\label{lemmaSumXinEq}
\sum_{k=0}^n k\, \chi_n(k) = n
\end{equation} 
This holds because:
\begin{align*}
\sum_{k=0}^n k\, \chi_n(k) &= n! \sum_{k=0}^n \frac{k}{(n-k)!\, n^k} = \frac{n!}{n^n} \sum_{k=0}^n \frac{n-k}{k!} n^k \\
&=n! \big(\frac{n}{n^n} \sum_{k=0}^n \frac{n^k}{k!} -\frac{n}{n^n}  \sum_{k=0}^n \frac{n^k}{k!} + \frac{1}{(n-1)!}\big)\\
&=\frac{n!}{(n-1)!} = n
\end{align*}

Next, we prove the proposition by induction. First, let us write $c_n(k)$ to denote the sequence of terms that satisfies: 
\begin{equation}\label{XiLimitSeq_1} 
\frac{p_n(0) s_0 + p_n(1) s_1+\dotsm + p_n(n) s_n}{\sum_{k=0}^n p_n(k)} = \sum_{k=0}^n c_n(k)\, a_k
\end{equation}
Again, we have $s_j = \sum_{k=0}^j a_k$. Our objective is to prove that $c_n(k)=\chi_n(k)$. To prove this by induction, we first note that $\sum_{k=0}^n p_n(k) = n$ by (\ref{lemmaSumXinEq}), and a base case is already established since $c_n(0)=1=\chi_n(0)$. Now, we note that: 
\begin{equation}\label{XiLimitSeq_2}
c_n(k)=\frac{p_n(k)+p_n(k+1) + \dotsm + p_n(n)}{n}
\end{equation}
For the inductive hypothesis, we assume that $c_n(k)=\chi_n(k)$ for $k<m$. To prove that this inductive hypothesis implies that $c_n(m)=\chi_n(m)$, we note by (\ref{XiLimitSeq_2}) that the following identity holds: 
\begin{equation}\label{XiLimitSeq_3}
c_n(m)=c_n(m-1)-\frac{p_n(m-1)}{n} =(1-\frac{m-1}{n}) \chi_n(m-1)=\chi_n(m)
\end{equation}
Here, we have used the inductive hypothesis and the original definition of $\chi_n(m)$. Therefore, we indeed have: 
\begin{equation}\label{XiLimitSeq_4}
\frac{p_n(0) s_0 + p_n(1) s_1+\dotsm + p_n(n) s_n}{\sum_{k=0}^n p_n(k)} = \sum_{k=0}^n \chi_n(k)\, a_k,
\end{equation}
which proves the statement of the proposition.
\end{proof}

Proposition \ref{XiLimitSeq} shows that the summability method $\chi$ is indeed {an averaging method} but it is different from the N\"orlund and the Ces\'aro means because the terms of the sequence $p_n(k)$ depend on $n$. In fact, whereas N\"orlund and Ces\'aro means are strictly weaker than the Abel summation method \citep{HardyDiverg}, we will show later that some divergent series, which are not Abel summable, have a well-defined value according to the summability method $\chi$. Therefore, the summability method $\chi$ is not contained in the N\"orlund and the Ces\'aro means. 

Next, Proposition \ref{XiLimitSeq} allows us to prove the following important statement.

\begin{corollary}\label{XiregStabLinear}
The summability method $\chi$ is linear and regular. 
\end{corollary}
\begin{proof} 
To show that the summability method is linear, we note using the basic laws of limits that: 
\begin{equation*}
\lim_{n\to\infty} \Big\{\sum_{j=0}^n \chi_n \big(\alpha_j + \lambda\, \beta_j\big) \Big\} = \lim_{n\to\infty} \Big\{\sum_{j=0}^n \chi_n(j)\, \alpha_j \Big\} +\lambda\, \lim_{n\to\infty} \Big\{\sum_{j=0}^n \chi_n(j)\, \beta_j \Big\}
\end{equation*}
Hence, $\chi:\,\mathbb{C}^\infty\to \mathbb{C}$ is a linear operator. To show regularity, we use the Toeplitz-Schur Theorem. The Toeplitz-Schur Theorem states that any matrix summability method $t_n=\sum_{k=0}^\infty A_{n,k} s_k \,(n=0,1,2,\ldots)$ in which $\lim_{n\to\infty} s_n$ is defined by $\lim_{n\to\infty} t_n$ is regular \emph{if and only if} the following three conditions hold \citep{HardyDiverg}: \index{regularity}\index{Toeplitz-Schur Theorem}\index{matrix summability}
\begin{enumerate}
\item $\sum_{k=0}^\infty |A_{n, k}| < H$, for all $n$ and some constant $H$ that is independent of $n$. 
\item $\lim_{n\to\infty} A_{n, k} = 0 \;$ for each $k$. 
\item $\lim_{n\to\infty} \sum_{k=0}^\infty A_{n,k} = 1$. 
\end{enumerate}

Using Proposition \ref{XiLimitSeq}, we see that the summability method $\chi$ is a matrix summability method characterized by $A_{n,k} = p_n(k)/n = k\,\chi_n(k)/n$. Because $A_{n,k}\ge 0$ and $\sum_{k=0}^\infty A_{n,k} = \sum_{k=0}^n A_{n,k} = 1$, both conditions 1 and 3 are immediately satisfied. In addition, for each fixed $k$,  $\lim_{n\to\infty} A_{n, k} = 0$, thus condition 2 is also satisfied. Therefore, the summability method $\chi$ is regular\footnote{In fact, because  $A_{n,k}\ge 0$, it is also \emph{totally regular} (see \citep{HardyDiverg} for a definition of this term).}.
\end{proof}

\section{Convergence}\label{sec::convergence}
In this section, we derive some properties related to the convergence of the summability method $\chi$. As is customary in the literature (cf. the \emph{Borel-Okada} principle \citep{HardyDiverg,tauberianTheory}), we analyze the conditions when $\chi$ evaluates the Taylor series expansion of the geometric series $\sum_{k=0}^\infty x^k$ to the value $(1-x)^{-1}$. We will focus on the case when $x\in\mathbb{R}$, and provide an asymptotic expression to the error term afterward. 

\subsection{The geometric series test}
We begin with the following elementary result. 
\begin{lemma}\label{exp_approx_bound_lemma}
For any $n\ge 1$ and all $x\in[0,\,n]$, we have $\big|(1-\frac{x}{n})^n-e^{-x} \big|\le \frac{1}{e\,n}$. 
\end{lemma} 
\begin{proof}
Define $g_n(x) = (1-x/n)^n-e^{-x}$ to be the approximation error of using $(1-x/n)^n$ for the exponential function $e^{-x}$. Then: 
\begin{align*}
g_n(0) &= 0 \\ 
|g_n(n)| &= e^{-n}\\
g'_n(x_0) = 0 \;\wedge\; x_0\in[0,\,n] \quad &\Leftrightarrow \quad e^{-x_0} = \big(1-\frac{x_0}{n}\big)^{n-1}
\end{align*}
In the last case, we have $|g(x_0)| = \frac{x_0}{n}\,e^{-x_0}\le \frac{1}{e\,n}$. Because $(e\,n)^{-1}\ge e^{-n}\ge 0$ for all $n\ge 1$, we deduce the statement of the lemma.
\end{proof}

\begin{proposition}\label{kappa_prop}
The $\chi$-sum of the infinite sequence $(x^k)_{k=0,1,2,\ldots}$ for $x\in\mathbb{R}$ is $(1-x)^{-1}$ if and only if $-\kappa < x < 1$, where $\kappa\approx 3.5911$ is the solution to $\kappa\,\log\kappa -\kappa = 1$. For values of $x\in\mathbb{R}$ that lie outside the open set $(-\kappa,\,1)$, the infinite sequence $(x^k)_{k=0,1,2,\ldots}$ is not $\chi$-summable.
\end{proposition} 
\begin{proof}
\textbf{(Case I)}. 
First, we consider the case when $x\ge 1$. Because every term in the sequence $x^k$ is non-negative, $\chi_n(k)\ge 0$, and $\lim_{n\to\infty} \chi_n(k)=1$ for any fixed $k\ge 0$, we have: 
\begin{equation*}
\lim_{n\to\infty}\;\Big\{ \sum_{k=0}^\infty \chi_n(k)\, x^k\Big\} = \infty, \quad\quad \text{ if } x\ge 1
\end{equation*} 
More generally, if an infinite sequence $(a_k)_{k=0,1,\ldots}$ is $\chi$-summable but $\sum_{k=0}^\infty a_k$ does not exist, then the sequence $(a_k)_{k=0,1,\ldots}$ must oscillate in sign infinitely many times. This is another way of stating that the summability method is \emph{totally-regular} \citep{HardyDiverg}. 

\textbf{(Case II)}. 
Second, we consider the case when $-1<x<1$. Because the geometric series converges in this domain and the summability method $\chi$ is regular (i.e. consistent with ordinary convergence) as proved in Corollary \ref{XiregStabLinear}, we obtain: 
\begin{equation*}
\lim_{n\to\infty}\;\Big\{ \sum_{k=0}^\infty \chi_n(k)\, x^k\Big\} = \frac{1}{1-x}, \quad\quad \text{ if } -1<x< 1
\end{equation*} 

\textbf{(Case III)}. Finally, we consider the case when $x\le -1$. To simplify notation, we will write $x=-z$, where $z\ge 1$. We have: 
\begin{align*}
\sum_{k=0}^n\,\chi_n(k)\,(-z)^k &= \sum_{k=0}^n\,\chi_n(k)\,\frac{(-z)^k}{k!}\;\int_0^\infty t^k\,e^{-t}\,\D t = \int_0^\infty \sum_{k=0}^n\,\chi_n(k)\frac{(-zt)^k}{k!}\,e^{-t}\;\D t\\ 
&= \int_0^\infty (1-\frac{tz}{n})^n\,e^{-t}\,\D t \\ 
&= \int_0^\frac{n}{z}\; (1-\frac{tz}{n})^n\,e^{-t}\,\D t  \quad + \quad \int_\frac{n}{z}^\infty \; (1-\frac{tz}{n})^n\,e^{-t}\,\D t
\end{align*}

Now, we consider each integral separately. First, we use the change of variable $w=1-\frac{tz}{n}$ to deduce that the second integral is given by: 
\begin{equation*}
\int_\frac{n}{z}^\infty \; (1-\frac{tz}{n})^n\,e^{-t}\,\D t = \frac{(n!)\,z^n\,e^\frac{n}{z}}{n^n}
\end{equation*}

For the first integral, on the other hand, we use Lemma \ref{exp_approx_bound_lemma} to obtain the bound:
\begin{align*}
\int_0^\frac{n}{z}  \; \Big|(1-\frac{tz}{n})^n\;-\;e^{-tz}\Big|\,e^{-t}\,\D t \;&\le\; \frac{1}{e\,n}\int_0^\frac{n}{z}  \; e^{-t}\,\D t \; \\
&= \; \frac{1-e^{-n/z}}{e\,n} \; \rightarrow\; 0\quad \text{ as } n\to\infty,
\end{align*}

Taking the limit as $n\to\infty$, we deduce that: 
\begin{equation*}
\lim_{n\to\infty}\;  \int_0^\frac{n}{z}\; (1-\frac{tz}{n})^n\,e^{-t}\,\D t = \lim_{n\to\infty} \int_0^\frac{n}{z} e^{-t(1+z)}\,\D t = \frac{1}{1+z} = \frac{1}{1-x}
\end{equation*} 

Therefore, in order for the summability method $\chi$ to evaluate the geometric series to the value $(1-x)^{-1}$, we must have:
\begin{equation*}
\lim_{n\to\infty}\;\big\{\frac{(n!)\,|x|^n\,e^\frac{n}{|x|}}{n^n}\big\} = 0 
\end{equation*}

Using Stirling's approximation \citep{Robbins1955}: 
\begin{equation*}
\frac{(n!)\,|x|^n\,e^\frac{n}{|x|}}{n^n}\,\sim  \Big(\frac{|x|\,e^\frac{1}{|x|}}{e}\Big)^n, 
\end{equation*}
which goes to zero as $n\to\infty$ only if $|x|<\kappa$.
\end{proof}

\subsection{Asymptotic Analysis} 
Next, we derive an asymptotic expression to the error term when the summability method $\chi$ is applied to a power series. To do this, we begin with the following lemma. 
\begin{lemma}\label{lemma441} 
Let $\sum_{k=0}^\infty a_k\,(x-x_0)^k$ be a power series for some function $f(x)$ that is analytic throughout the domain $[x_0,\,x]$. Then: 
\begin{equation}\label{lemma441Eq} 
\lim_{n\to\infty} \Big\{\sum_{k=1}^n \chi_n(k) \Big[f(x)-\sum_{j=0}^{k-1} \frac{f^{(j)}(x_0)}{j!} (x-x_0)^j\Big] \Big\}= (x-x_0)\,f'(x),
\end{equation}
provided that $\sum_{k=0}^n \chi_n(k) a_k\,(x-x_0)^k$ converges uniformly to $f(x)$ in the domain $[x_0,\,x]$ as $n\to\infty$. 
\end{lemma}
\begin{proof}
It is straightforward to see that if the limit in (\ref{lemma441Eq}) exists, then (\ref{lemma441Eq}) must hold. We will prove this formally. To see this, define $g_f$ to be a functional of $f(x)$ that is given by: 
\begin{equation}\label{lemma441_1} 
g_f=\lim_{n\to\infty} \Big\{\sum_{k=1}^n \chi_n(k) \Big[f(x)-\sum_{j=0}^{k-1} \frac{f^{(j)}(x_0)}{j!} (x-x_0)^j\Big] \Big\}
\end{equation}
Now, we differentiate both sides with respect to $x$, which yields: 
\begin{align*}\label{lemma441_2} 
\frac{\D}{\D x} g_f&=\lim_{n\to\infty} \Big\{\sum_{k=1}^n \chi_n(k) \Big[ \frac{f^{(k)}(x_0)}{(k-1)!} (x-x_0)^{(k-1)}+f'(x)-\sum_{j=0}^{k-1} \frac{f^{(j+1)}(x_0)}{j!} (x-x_0)^j\Big] \Big\}\\
&=g_{f'} + \lim_{n\to\infty} \sum_{k=1}^n \chi_n(k) \frac{f^{(k)}(x_0)}{(k-1)!} (x-x_0)^{(k-1)}\\
&=g_{f'} + f'(x)
\end{align*}
Therefore, we have: 
\begin{equation}\label{lemma441_3}
\frac{\D}{\D x} g_f = f'(x)+ g_{f'}
\end{equation} 
Since $g_f(x_0)=0$, the solution is given by: 
\begin{equation}\label{lemma441_4}
g_f(x) = (x-x_0)\,f'(x)
\end{equation} 
The above derivation assumes that the limit exists. To prove that (\ref{lemma441Eq}) holds under the stated conditions, we note that $f(x)-\sum_{j=0}^{k-1} \frac{f^{(j)}(x_0)}{j!} (x-x_0)^j$ is the error term of the Taylor series expansion, which is \emph{exactly} given by: 
\begin{equation}\label{lemma441_5}
f(x)-\sum_{j=0}^{k-1} \frac{f^{(j)}(x_0)}{j!} (x-x_0)^j = \int_{x_0}^x \frac{f^{(k)}(t)}{(k-1)!} (x-t)^{k-1} \,dt
\end{equation} 
Upon using the last expression, we have that: 
\begin{align}\label{lemma441_6}
\nonumber g_f &= \lim_{n\to\infty} \sum_{k=1}^{n} \chi_n(k) \int_{x_0}^x \frac{f^{(k)}(t)}{(k-1)!} (x-t)^{k-1} \,dt \\
&= \int_{x_0}^x  \lim_{n\to\infty}\sum_{k=1}^{n} \chi_n(k) \frac{f^{(k)}(t)}{(k-1)!} (x-t)^{k-1} \,dt
\end{align} 

Here, exchanging sums with integrals is justifiable when uniform convergence holds. However:
\begin{equation}\label{lemma441_7}
 \lim_{n\to\infty}\sum_{k=1}^{n} \chi_n(k) \frac{f^{(k)}(t)}{(k-1)!} (x-t)^{k-1} = f'(x), \\ \text{ for all } t\in[x_0, x]
\end{equation} 
Plugging (\ref{lemma441_7}) into (\ref{lemma441_6}) yields the desired result.
\end{proof}

\begin{theorem}\label{theorem441} 
Let $\hat f_n(x) = \sum_{k=0}^n\,\chi_n(k)\,\frac{f^{(k)}(x_0)}{k!}\,(x-x_0)^k$. Then, under the stated conditions of Lemma \ref{lemma441}, the error as $n\to\infty$ is asymptotically given by:
\begin{equation}\label{theorem441ErrorTerm}
f(x)-\hat f_n(x) \;\sim\; \frac{f^{\prime\prime}(x) (x-x_0)^2}{2n}
\end{equation}
\end{theorem}
\begin{proof}
Because the function $f(x)$ is analytic at every point in the domain $[x_0, x]$, define $\epsilon>0$ to be the distance between the set $[x_0, x]$ and the nearest singularity point of $f(x)$. More formally, let $\mathcal{B}(z,\,r)\subset \mathbb{C}$ be the open ball of radius $r$ centered at $z$, and define: 
\begin{equation*}
\epsilon = \sup\; \{r\;|\;\forall z\in[x_0,\,x]: f \text{ is analytic in } \mathcal{B}(z,r) \}
\end{equation*}

In our construction of the summability method $\chi$ in Section \ref{sum_taylor_derivation}, we have used the following linear approximations, where $f_j^{(k)}$ is a shorthand for $f^{(k)}(x_0+jh)$, and $h$ is a small step size: 
\begin{equation}\label{XiErrTherm_1}
f_j^{(k)} \approx f_{j-1}^{(k)} + f_{j-1}^{(k+1)} h 
\end{equation}

Because the distance from $x_0+jh$ to the nearest singularity point of $f(x)$ is at least $\epsilon$, then selecting $h<\epsilon$ or equivalently $n>\frac{x-x_0}{\epsilon}$ implies that the error term of this linear approximation is \emph{exactly} given by (\ref{XiErrTherm_2}). This follows from the classical result in complex analysis that an analytic function is equal to its Taylor series representation within its radius of convergence, where the radius of convergence is, at least, equal to the distance to the nearest singularity.
\begin{equation}\label{XiErrTherm_2}
E_j^k = f_j^{(k)} - f_{j-1}^{(k)} - f_{j-1}^{(k+1)} h= \sum_{m=2}^\infty \frac{f_{j-1}^{(k+m)}}{m!} h^m
\end{equation}

Since we seek an asymptotic expression as $n\to\infty$, we assume that $n$ is large enough for (\ref{XiErrTherm_2}) to hold (i.e. $n>(x-x_0)/\epsilon$). Because higher-order derivatives at $x_0$ are computed exactly, we have $E_0^k =0$. Now, the linear approximation method was applied recursively in our construction of $\chi$ in Section \ref{sum_taylor_derivation}. Visually speaking, this repeated process mimics the expansion of a binary pyramid, depicted in Figure \ref{Fig441Pyramid}, whose nodes $(j, k)$ correspond to the linear approximations given by (\ref{XiErrTherm_1}) and the two children of each node $(j, k)$ are given by $(j-1, k)$ and $(j-1, k+1)$ as stated in that equation. It follows, therefore, that the number of times the linear approximation in (\ref{XiErrTherm_1}) is used for a fixed $n$ is equal to the number of paths from the root to the respective node in the binary pyramid, where the root is $(n, 0)$. It is a well-known result that the number of such paths is given by $\binom{n-j}{k}$.
\begin{figure} [t] 
\centering
\includegraphics{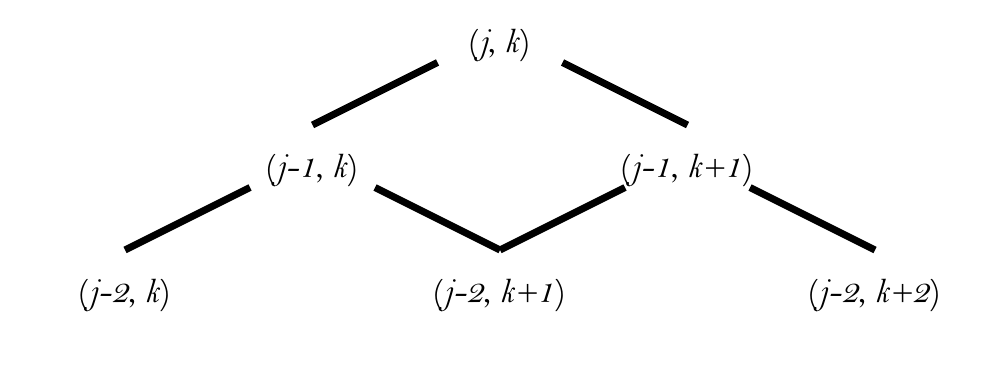}
\caption[A depiction of the recursive proof of the summability method $\Xi$]{A depiction of the recursive proof of the summability method in Proposition \ref{claim411}}
\label{Fig441Pyramid}
\end{figure}

Consequently, the error term of the $\chi$ summability method is given by:
\begin{equation}\label{XiErrTherm_3}
f(x)-\hat f_n(x) =\sum_{j=1}^n \binom{n-j}{0} E_j^0 +  \sum_{j=1}^{n-1} \binom{n-j}{1} E_j^1 h+  \sum_{j=1}^{n-2} \binom{n-j}{2} E_j^2 h^2\dotsm
\end{equation}
Define $e_k$ to be a weighted average of the errors $E_j^k$ that is given by: 
\begin{equation}\label{XiErrTherm_4}
e_k = \frac{\sum_{j=1}^{n-k} \binom{n-j}{k} E_j^k}{\sum_{j=1}^{n-k} \binom{n-j}{k}} 
\end{equation}
Then, we have:
\begin{equation}\label{XiErrTherm_5}
f(x)-\hat f_n(x) = e_0 \sum_{j=1}^n \binom{n-j}{0} +  e_1 \sum_{j=1}^{n-1} \binom{n-j}{1} h+  e_2 \sum_{j=1}^{n-2} \binom{n-j}{2} h^2 \dotsm
\end{equation}
Now, we make use of the identity $\sum_{j=1}^{n-k} \binom{n-j}{k} = \binom{n}{k+1}$, and substitute $h= \frac{x-x_0}{n}$, which yields: 
\begin{equation}\label{XiErrTherm_6}
f(x)-\hat f_n(x) = \frac{n}{x-x_0} \sum_{k=1}^n \chi_n(k) \frac{e_{k-1}}{k!} (x-x_0)^k
\end{equation}
Interestingly, the terms $\chi_n(k)$ appear again in the approximation error. Now, define $e_k^*$ by:
\begin{equation}\label{XiErrTherm_7}
e_k^* = \frac{e_k}{h^2} =  \sum_{j=1}^{n-k} \frac{\binom{n-j}{k}}{\binom{n}{k+1}} \sum_{m=0}^\infty \frac{f_{j-1}^{(k+m+2)}}{(m+2)!} h^{m}
\end{equation}
This yields:
\begin{equation}\label{XiErrTherm_8}
f(x)-\hat f_n(x) = \frac{x-x_0}{n} \sum_{k=1}^n \chi_n(k) \frac{e_{k-1}^*}{k!} (x-x_0)^k
\end{equation}

Next, we examine the weighted average error terms $e_k^*$ at the limit $n\to\infty$.  Because we desire an asymptotic expression of $f(x)-f_n(x)$ as $n\to\infty$ up to a first order approximation, we can safely ignore any $o(1)$ errors in the asymptotic expressions for $e_{k-1}^*$ in (\ref{XiErrTherm_8}). 

First, we note that the expression for $e_k^*$ given by (\ref{XiErrTherm_7}) is \emph{exact} and that the infinite sum converges because $n$ is assumed to be large enough such that $h$ is within the radius of convergence of the series (i.e. satisfies $h < \epsilon$). Therefore, it follows that $e_k^*$ is asymptotically given by (\ref{XiErrTherm_9}), where the expression is asymptotic as $n\to\infty$. This follows from the fact that a Taylor series expansion around a point $x_0$ is an asymptotic expression as $x\to x_0$ in the Poincar\'e sense if $f(x)$ is regular at $x_0$. 
\begin{equation}\label{XiErrTherm_9}
e_k^* \sim \sum_{j=1}^{n-k} \frac{\binom{n-j}{k}}{\binom{n}{k+1}} \big[\frac{f_{j-1}^{(k+2)}}{2} + o(1)\big]
\end{equation}
Because $p_k(j;\,n)=\binom{n-j}{k}/\binom{n}{k+1}$ is a normalized discrete probability mass function, and from the earlier discussion, the $o(1)$ term can be safely ignored. This gives us: 
\begin{equation}\label{XiErrTherm_9_2}
e_k^* \sim \sum_{j=1}^{n-k} \frac{\binom{n-j}{k}}{\binom{n}{k+1}} \frac{f_{j-1}^{(k+2)}}{2}
\end{equation}

To come up with a more convenient form, we note that the discrete probability mass function $p_k(j;\,n)=\binom{n-j}{k}/\binom{n}{k+1}$ approaches a probability density function at the limit $n\to\infty$ when $\frac{j}{n}$ is fixed at a constant $z$. Using Stirling's approximation \citep{Robbins1955}, such a probability density is given by: 
\begin{equation}\label{XiErrTherm_10}
\rho_k(z)=(1+k)(1-z)^k,\quad\quad 0\leq z \leq 1
\end{equation}
For example, if $k=0$, the probability density function is uniform as expected.

Upon making the substitution $z=j/n$, the discrete probability mass function $p_k(j;\,n)$ can be approximated by $\rho_k(z)$ whose approximation error vanishes at the limit $n\to\infty$. Using $\rho_k(z)$, $e_k^*$ is asymptotically given by: 
\begin{equation}\label{XiErrTherm_11}
e_k^* \sim \frac{1+k}{2} \int_0^1 (1-t)^k f^{(k+2)} \big(x_0 + t(x-x_0)\big)\,\D t
\end{equation}
Doing integration by parts yields the following recurrence identity:
\begin{equation}\label{XiErrTherm_12}
e_k^* \sim -\frac{1+k}{2(x-x_0)} f^{(k+1)}(x_0) + \frac{1+k}{x-x_0} e_{k-1}^* 
\end{equation}
Also, we have by direct evaluation of the integral in (\ref{XiErrTherm_10}) when $k=0$:
\begin{equation}\label{XiErrTherm_13}
e_0^* \sim \frac{f'(x)-f'(x_0)}{2(x-x_0)} 
\end{equation}
Combining both (\ref{XiErrTherm_12}) and (\ref{XiErrTherm_13}), we obtain the convenient expression:
\begin{equation}\label{XiErrTherm_14}
e_k^* \sim \frac{1}{2}  \frac{(1+k)!}{(x-x_0)^{k+1}} f'(x) - \frac{1}{2} \frac{(1+k)!}{(x-x_0)^{k+1}} \sum_{m=0}^k \frac{f^{(m+1)}(x_0)}{m!} (x-x_0)^m
\end{equation}
Plugging (\ref{XiErrTherm_14}) into (\ref{XiErrTherm_8}) yields: 
\begin{equation}\label{XiErrTherm_15} 
f(x)-\hat f_n(x) \sim \frac{x-x_0}{2n} \sum_{k=1}^n \chi_n(k) \Big[f'(x)-\sum_{m=0}^{k-1} \frac{f^{(m+1)}(x_0)}{m!} (x-x_0)^m \Big]
\end{equation}
Using Lemma \ref{lemma441}, we arrive at the desired result:
\begin{equation}\label{XiErrTherm_Final} 
f(x)-\hat f_n(x) \sim \frac{f^{(2)}(x)\, (x-x_0)^2}{2n}
\end{equation}
\end{proof}

To test the asymptotic expression given by Theorem \ref{theorem441}, suppose we evaluate the geometric series $f(x) = \sum_{k=0}^\infty x^k$ at $x=-2$ using $n=40$. Then, the error term by direct application of the summability method $\chi$ is 0.0037 (up to four decimal places). The expression in Theorem \ref{theorem441} predicts a value of 0.0037, which is indeed accurate up to 4 decimal places. Similarly, if we apply the summability method $\chi$ to the Taylor series expansion $f(x)=\log{(1+x)} = \sum_{k=1}^\infty \frac{(-1)^{k+1}}{k}$ at $x=3$ using $n=30$, then the error term is -0.0091 (up to four decimal places) whereas Theorem \ref{theorem441} estimates the error term to be -0.0094.

The asymptotic expression for the error term given in Theorem \ref{theorem441} presents an interesting insight. Specifically, we can determine if the summability method converges to a value $f(x)$ from above or from below depending on whether the function $f$ is concave or convex at $x$. This conclusion becomes intuitive if we keep in mind that the summability method $\chi$ uses  first-order linear approximations, in a manner that is quite similar to the Euler method. Thus, at the vicinity of $x$, first order linear approximation overestimates the value of $f(x)$ if $f$ is concave at $x$ and underestimates it if $f$ is convex.

\section{Examples} 
In Section \ref{sec::convergence}, we proved that the summability method $\chi$ \lq\lq correctly'' evaluates the geometric series $\sum_{k=0}^\infty x^k$ in the domain $-\kappa<x<1$, where $\kappa\approx 3.5911$ is the solution to $\kappa\log\kappa -\kappa = 1$. Because $\kappa>1$, the summability method is strictly more powerful than the Abel summation method and all the N\"orlund and the Ces\'aro means. In this section, we present two additional examples that demonstrate how the summability method $\chi$ indeed assigns \lq\lq correct'' values to divergent sums. 

We begin with the following example.

\begin{example}\normalfont
It can be shown using various arguments (see for instance the simple derivations in \citep{alabdulmohsinSummabilityCalculus}) that the following assignments are valid, e.g. in the sense of analytic continuation: 
\begin{equation}\label{example1_exact}
\sum_{k=0}^\infty (-1)^{k} H_{k+1} = \frac{\log 2}{2},\; \sum_{k=0}^\infty (-1)^{k}\,\log (1+k) = \log\sqrt{\frac{2}{\pi}},\; \sum_{k=0}^\infty (-1)^{k}=\frac{1}{2}
\end{equation}
These expressions can be easily verified using the summability method $\chi$, which gives the following values when $n=100$: 
\begin{equation*}
\sum_{k=0}^\infty (-1)^{k} H_{k+1} \approx 0.3476,\quad \sum_{k=0}^\infty (-1)^{k}\,\log (1+k) \approx -0.2261,\quad \sum_{k=0}^\infty (-1)^{k}\approx 0.4987
\end{equation*}
The latter values agree with the exact expressions in (\ref{example1_exact}) up to two decimal places. A higher accuracy can be achieved using a larger value of $n$.

However, a famous result, due to Euler, states that the harmonic numbers $H_n = \sum_{k=1}^n\,(1/k)$ are asymptotically related to $\log{n}$ by the relation $H_n\sim \log n + \gamma$, where $\gamma\approx 0.577$ is the Euler constant. This implies that the alternating series
\begin{equation*}
\sum_{k=0}^\infty (-1)^{k} \big(H_{k+1}-\log (k+1) -\gamma\big)
\end{equation*} 
converges to some value $V\in\mathbb{R}$. It is relatively simple to derive an \emph{approximate} value of $V$ using, for example, the Euler-Maclaurin summation formula. However, regularity and linearity of the summability method $\chi$ can be employed to deduce the \emph{exact} value of the convergent sum. Specifically, we have:
\begin{align*}
V & = \sum_{k=0}^\infty (-1)^{k} \big(H_{k+1}-\log (k+1) -\gamma\big) && \text{\emph{(by definition)}}\\ 
&= \lim_{n\to\infty}\; \sum_{k=0}^n \chi_n(k)\;(-1)^{k} \big(H_{k+1}-\log (1+k) -\gamma\big) && \text{\emph{(by regularity)}}\\ 
&= \Big[\lim_{n\to\infty}\; \sum_{k=0}^n \chi_n(k)\;(-1)^{k} H_{k+1}\Big] \, - \Big[\lim_{n\to\infty}\; \sum_{k=0}^n \chi_n(k)\;&&(-1)^{k} \log(k+1)\Big]\\ &\quad\quad-\gamma\,\Big[\lim_{n\to\infty}\; \sum_{k=0}^n \chi_n(k)\;(-1)^{k}\Big]&& \text{\emph{(by linearity)}}
\end{align*} 
Plugging in the exact values in (\ref{example1_exact}), we deduce that: 
\begin{equation*}
\sum_{k=0}^\infty (-1)^{k} \big(H_{k+1}-\log (k+1) -\gamma\big) = \frac{\log\pi-\gamma}{2}
\end{equation*}
Hence, we used the values (a.k.a. \emph{antilimits}) of divergent series to determine the limit of a convergent series. 
\end{example} 

\begin{example}\normalfont
In our second example, we consider the series $f(x)=\sum_{k=0}^\infty B_k\,x^k$, where $B_k = (1,\frac{1}{2},\frac{1}{6},0,\ldots)$ are the Bernoulli numbers. Here, $f(x)$ diverges everywhere except at the origin. Using the Euler-Maclaurin summation formula, it can be shown that $f(x)$ is an asymptotic series to the function $h(x) = \frac{1}{x}\psi^{\prime\prime}(1/x) + x$ as $x\to 0$, where $\psi(x) = \log (x!)$ is the log-factorial\footnote{In MATLAB, the function is given by the command: \ttfamily{1/x*psi(1,1/x+1)+x}}. Hence, we can indeed contrast the values of the divergent series, when interpreted using the summability method $\chi$, with the values of $h(x)$. However, because $B_k$ increases quite rapidly, the infinite sum $\sum_{k=0}^\infty B_k\,x^k$ is not $\chi$-summable \emph{per se} but it can be approximated using small values of $n$, nevertheless. 

Table \ref{Tab43Bern} lists the values of $\sum_{k=0}^n\,\chi_n(k)\,B_k\,x^k$ for different choices of $n$ and $x$. As shown in the table, the values are indeed close to the \lq\lq correct" values, given by $h(x)=\frac{1}{x}\psi^{\prime\prime}(1/x) + x$, despite the fact that the finite sums $\sum_{k=0}^n\,B_k\,x^k$ bear no resemblance with the generating function $h(x)$. Note, for instance, that if $n=30$, then  $\sum_{k=0}^n\,B_k \approx 5.7\times 10^8$, whereas $\sum_{k=0}^n\,\chi_{n}(k)B_k\approx 1.6425$, where the latter figure is accurate up to two decimal places. This agrees with Euler's conclusion that $\sum_{k=0}^\infty B_k = \zeta_2 = \pi^2/6$  \citep{Peng2002}. 

Most importantly, whereas the summability method $\chi$ is shown to be useful in this example, even when the series $\sum_{k=0}^\infty B_k\,x^k$ diverges quite rapidly, other classical summability methods such as the Abel summation method and the Mittag-Leffler summability method cannot be used in this case. Hence, the summability method $\chi$ can be quite useful in computing the values of divergent series where other methods fail. 

\begin{table} 
\caption {For every chioce of $x$ and $n$, the value in the corresponding cell is that of $\sum_{k=0}^n\,\chi_n(k)\,B_k\,x^k$. The exact value in the last row is that of the generating function $\frac{1}{x}\psi^{\prime\prime}(1/x) + x$, where $\psi(x) = \log (x!)$ is the log-factorial function.}\label{Tab43Bern}
\centering\begin{tabular} { |l|c|c|c|c|c|c|c|}
\hline 
  &  $x=-1$ & $x=-0.7$ & $x=-0.2$& $x=0$ & $x=0.2$ & $x=0.7$ & $x=1$\\[7pt]
\hline
$n=20$  &  0.6407 & 0.7227  & 0.9063 & 1.000 & 1.1063 &  1.4227 & 1.6407\\[7pt]
$n=25$  &  0.6415 & 0.7233  & 0.9064 & 1.000 & 1.1064 &  1.4233 & 1.6415\\[7pt]
$n=30$  &  0.6425 & 0.7236  & 0.9064 & 1.000 & 1.1064 &  1.4236 & 1.6425\\[7pt]
\hline
\textbf{\;Exact}  &  0.6449 & 0.7255  & 0.9066 & 1.000 & 1.1066 & 1.4255 & 1.6449\\[7pt]
\hline
\end{tabular}
\end{table}
\end{example}

\section{Conclusion}
In this paper, a new analytic summability method for divergent series is introduced. We showed how it arose quite naturally in the study of local polynomial approximations of analytic functions, and provided an asymptotic expression to its error term whenever it converged. We also proved that it was both linear and regular, and that it could be interpreted as an averaging method. Finally, we  demonstrated how the proposed summability method could be quite useful where other methods failed, such as for asymptotic series that diverged quite rapidly. 
%
%



\section*{Bibliography}
\bibliographystyle{elsarticle-num-names}      
\bibliography{summ_method}   

%
%

\end{document}